\documentclass[12pt, reqno]{amsart}
\usepackage{amssymb}
\input xypic
\xyoption{all}

\newtheorem{thm}{Theorem}
\newtheorem{pro}[thm]{Proposition}
\newtheorem{cor}[thm]{Corollary}
\newtheorem{lem}[thm]{Lemma}

\newtheorem{con}{Conjecture}
\theoremstyle{remark}
\newtheorem{rem}[thm]{Remark}

\newtheorem{exa}[thm]{Example}

\newcommand{\calf}{{\mathcal F}}

\newcommand{\ideala}{{\mathfrak a}}

\newcommand{\idealm}{{\mathfrak m}}

\newcommand{\aut}{{\rm Aut}}
\newcommand{\autp}{{\rm Aut}_\C(\C[x,y])}

\newcommand{\bir}{{\rm Bir}}

\newcommand{\der}{{\rm Der}}
\newcommand{\derp}{{\rm Der_\C}(\C[x,y])}

\newcommand{\C}{{\mathbb K}}
\newcommand{\CC}{{\mathbb C}}
\newcommand{\K}{{\mathbb K}}

\newcommand{\tor}{\xymatrix{\ar@{-->}[r]&}}

\begin{document}

\begin{center}
\author{Lu\'is Gustavo Mendes and Ivan Pan}
\end{center}

\thanks{Research of L.G.Mendes  was partially supported by Pronex/CNPq of Brazil}
\thanks{Research of I. Pan was partially supported by ANII and PEDECIBA of Uruguay}

\title[automorphisms commuting with simple derivations]{On plane polynomial automorphisms commuting with simple derivations}
\maketitle


\section{introduction}\label{sec1}
 
We denote by $\C[x,y]$ the ring of polynomials in two variables over an algebraically closed field $\C$ of characteristic zero, and by $\autp$ the group consisting of $\C$-linear automorphisms of $\C[x,y]$. 

A \emph{derivation} of $\C[x,y]$ is a $\C$-linear map $D:\C[x,y]\to\C[x,y]$ such that $D(fg)=gD(f)+fD(g)$ for all $f,g\in\C[x,y]$. We denote by $\derp$ the $\C-$linear vector space of derivations of $\C[x,y]$; this is, in fact, a $\C[x,y]$-module.
  
If $D\in\derp$ we denote by $\aut(D)$ the subgroup of $\autp$ consisting of automorphisms commuting with $D$. In other words, $\autp$ acts on $\derp$ by conjugation as
\[\rho\cdot D=\rho D\rho^{-1},\quad    \rho\in\autp,\quad   D\in\derp,\]
and $\aut(D)$ is the isotropy of $D$ with respect to that action. 

A derivation $D$ is said to be \emph{simple} if it does not stabilize a nontrivial ideal. 

Finally, a derivation of the form  
\[D=\partial_x+(ay+b)\partial_y, \ a,b\in\C[x]\]
is said to be a \emph{Shamsuddin derivation}. In \cite{Sha} there is 
a criterion which allows to decide whether $D$ is simple or not,  depending if the 
equation $D(h)=ah+b$ has 
no solution $h\in\C[x]$ or admits such a solution, respectively.  It follows trivially that a Shamsuddin derivation with $a=0$ is not simple.

In \cite[Thm.6]{Bal} it is proved that $\aut(D)$ is trivial if $D$ is a simple  Shamsuddin  derivation.  

The main result of this work, proved in the next section, is the following:
\begin{thm}\label{thm2}
If  $D$ is a simple derivation, then $\aut(D)$ is trivial.
\end{thm}

In the last section we describe the isotropy group of a Shamsuddin derivation with $a\neq 0$ (Proposition \ref{pro1}) and prove, as an application, the following:

 \begin{thm}\label{thm1} Let  $D$ be  a Shamsuddin  derivation with $a\neq 0$. Then $D$ is simple if and only if $\aut(D)=\{id\}$. 
\end{thm}


Let us now make some remarks in order to connect our results with other topics.
 
First, assume that $\C=\CC$ is the complex number field. Denote by $\calf_D$ the algebraic  foliation on $\CC^2$, with isolated singularities, defined by a derivation $D$.  Note that $D$ is simple if and only if $\calf_D$ does not have  singularities nor  algebraic  leaves.

An element $\rho\in\aut(\CC[x,y])$ is determined by $f:=\rho(x)$ and $g:=\rho(y)$. The associated map $R:\CC^2\to \CC^2$ defined 
by $(t,u)\mapsto ( f(t,u),g(t,u))$ is an automorphism in the algebraic geometric sense. If $\rho$ commutes 
with $D$, then $R$ maps a leaf of $\calf_D$ onto another leaf. 

Consider the group $\bir(\calf_D)$ made up of birational maps $\CC^2\tor\CC^2$ which stabilize $\calf_D$  (\cite{SP}, \cite{CaFa}). We deduce that there is a 
natural one to one  homomorphism $\varphi_D:\aut(D)\to  \bir(\calf_D)$. As shown in the example below, $\varphi_D$ is not in general surjective. Moreover, there are foliations for which $\bir(\calf_D)$ is infinite (see \emph{loc. cit.}). 

We state a conjecture that we shall consider in a forthcoming paper:

\begin{con}
If $D$ is simple, then  $\bir(\calf_D)$ is finite.
\end{con} 

\smallskip

\begin{exa}
Consider the derivation 
\[D=(1+xy+x^3)\partial_{x}+x(1+xy)\partial_{y}.\] 
 
If $j=e^{2\pi{\bf i}/3}$ is a primitive cubic root of the unity,  the linear automorphisms $(x,y)\mapsto (j^2 x,j y), (x,y)\mapsto (j x,j^2 y)$ belong to $\aut(\calf_D)\subset\bir(\calf_D)$ but not to $\aut(D)$. On the other hand, the birational change of coordinates $u=x/y, v=1/y$ maps $y^{-1}D$ into 
\[D'=(u+v^2-u^2v)\partial_u-u(u+v^2)\partial_v.\]
By \cite[Proposition 1.3]{CDGM} the associated foliation $\calf_{D'}=\calf_D$ has no invariant algebraic curves. Since $D$ does not vanish in $\CC^2$ and $y=0$ is not stable by $D$ we deduce that it is a simple derivation; in particular $\aut(D)$ is trivial.   
\end{exa}

Finally, the notion of simplicity for a derivation may be extended to an arbitrary $\K$-algebra in a natural form. Following the pioneering works of Seidenberg and Hart (resp. \cite{Sei} and \cite{Hart75}) one knows that the local ring of an algebraic variety at a point is regular if and only if it admits a simple derivation. In particular, a polynomial ring $A=\K[x_1,\ldots,x_n]$ over $\K$ admits a simple derivation at each localization by a maximal ideal and there are criteria to decide if a derivation in such a localization is simple (\cite{BLL}). However, there are no simplicity criteria for derivations in $A$, besides Shamsuddin's criterion and a more restrictive one given by Y. Lequain in \cite{Leq07}, where an algorithm is given to decide whether a certain type of Shamsuddin derivation is simple. Furthermore, even for $\K=\CC$ and $n=2$ few simple derivations are known besides the ones given by such a criterion (see for example \cite{Now08}, \cite{Col08}, \cite{Sa12}, \cite{SK} and references therein). Related to these facts we have the following natural question: 

How far a derivation $D\in\der_{\K}(R)$ with $\aut(D)=\{id\}$ is from being simple ? 

\smallskip

\noindent{\bf Acknowledgment}: The second author thanks J\'er\'emy Blanc for useful conversations.
         
\section{Isotropy of simple derivations}\label{sec3}

We start with some general remarks relating $D$, not necessarily simple, with fixed points of elements in $\aut(D)$. Let us first recall some notions.

Consider a derivation $D=a\partial_x+b\partial_y$ of $\C[x,y]$ and fix a point $p=(p_1,p_2)\in\C^2$; denote by $\idealm$ the maximal ideal of $\C[x,y]$ generated by $x-p_1, y-p_2$. If $a$ or $b$ do not belong to $\idealm$, then there exists a unique (convergent if $\K=\CC$; see reference below for the general case) power series $\phi(t),\psi(t)\in\C[[t]]$ such that $\partial_t\phi(t)=a(\phi(t),\psi(t)),  \partial_t\psi(t)=b(\phi(t),\psi(t))$, with $\phi(0)=p_1, \psi(0)=p_2$. Note that $a\not\in\idealm$ or $b\not\in\idealm$ is equivalent to $D(\idealm)\not\subset\idealm$.

A straightforward calculation shows that the $\C$-homomorphism $\varphi:\C[x,y]\to\C[[t]]$ defined by mapping $x\mapsto \phi(t), y\mapsto \psi(t)$ is the unique $\C$-homomorphism such that $\partial_t\varphi=\varphi D$ and $\varphi^{-1}(t\,\C[[t]])=\idealm$. We say that such a $\varphi$ is a \emph{solution of $D$  passing through} $\idealm$ (see \cite[Thm. 7c]{BP} for other properties of this notion).

\begin{lem}\label{lem1}
Let $D$ be a derivation and let $\rho\in\aut(D)$. If $\rho$ fixes a maximal ideal $\idealm$ which is not stable under $D$, then there exists a principal ideal $\ideala\subset\idealm$ such that:

a) $\ideala$ is stable under $D$ and fixed by $\rho$.

b) $\rho$ induces the identity map in $\C[x,y]/\ideala$.
\end{lem}

\begin{proof}
Let $\idealm$ be a maximal ideal such that $\rho(\idealm)=\idealm$. Consider the unique solution $\varphi:\C[x,y]\to \C[[t]]$ of $D$ passing through $\idealm$. Note that $D$ stabilizes $\ker(\varphi)$; indeed, if $\varphi(g)=0$, then $\varphi(D(g))=\partial_t\varphi(g)=0$.

Since $\partial_t\varphi\rho=\varphi D\rho=\varphi\rho D$ we know that $\varphi\rho$ is a solution of $D$ passing through $\rho^{-1}(\idealm)=\idealm$. Hence, $\varphi\rho=\varphi$ from which it follows $\rho(\ker(\varphi))=\ker(\varphi)$. 

If $g\in\C[x,y]\setminus\ker(\varphi)$, then $\varphi(\rho(g)-g)=0$; so $\rho$ induces the identity map in $\C[x,y]/\ker(\varphi)$. 

Since $\ker(\varphi)$ is a prime ideal strictly contained in $\idealm$ its height is $0$ or 1. Hence $\ker(\varphi)$ is principal, and the proof follows by taking $\ideala=\ker(\varphi)$. 

\end{proof}

\begin{rem}
First, note that in part a) of the lemma above the ideal $\ideala$ may be trivial, as it happens, for example, when $D$ is a simple derivation. Second, if $\rho=id$ the part b) of that lemma is not informative.  Finally, note that the ideal $\ideala$ is uniquely determined by $\idealm$ since $\varphi$ is. 
\end{rem}

\begin{cor}\label{cor2}
Let $D$ be a simple derivation and let $\rho\in\aut(D)$. If $\rho\neq id$, then $\rho^N$ does not fix maximal ideals for any $N$. In particular, either $\rho$ has infinite order or $\rho=id$. 
\end{cor}

\begin{proof}
Suppose that $\rho^N$ fixes a maximal ideal $\idealm$ with $N>1$. Since $\rho^N\in\aut(D)$,  Lemma \ref{lem1}a) implies there exists an ideal $\ideala\subset \idealm$, stable under $D$. Since $D$ is simple $\ideala=(0)$. Lemma \ref{lem1}b) implies $\rho^N=id$. 

On the other hand, it is well known that an automorphism of finite order is conjugate to a linear one (it is an easy consequence of van der Kulk Theorem: see for example \cite[\S2]{KS} and note that the theorem holds over $\K$), hence it fixes a maximal ideal. Then we apply again Lemma  \ref{lem1} to obtain $\rho=id$.
\end{proof}

\begin{lem}\label{lem7}
Let $D$ be a simple derivation and let $\rho\in\aut(D)$. If $\rho$ stabilizes the ideal generated by $x\in\C[x,y]$, then $\rho=id$.  
\end{lem}
\begin{proof}
By Corollary \ref{cor2} it suffices to prove that $\rho$ has finite order.

Assume, by contradiction, that $\rho$ has infinite order. 

There exists $\alpha\in\C^*$ such that $\rho(x)=\alpha x$. Write $\rho(y)=\sum_{j=0}^m g_j y^j$, for $g_j\in\C[x]$, $j=0,\ldots,m$. Since the Jacobian of $(\rho(x), \rho(y))$ belongs to $\C^*$, we obtain $g_1=\beta\in\C^*$ and $g_j=0$ for $j>1$, that is $\rho(y)=g(x)+\beta y$ for some $g\in\C[x]$. Moreover, since $\rho$ does not fix maximal ideals, we get $\beta=1$ and $g(0)\neq 0$; in particular $g\neq 0$.

Now assume that $\alpha$ is not a root of the unity and write $D=a\partial_x+b\partial_y$, with $a,b\in\C[x,y]$. Since $\rho D(x)=D\rho(x)$ we have $a(\alpha x,g(x)+y)=\alpha a(x,y)$. In other words, if $a=\sum_{i=0}^n a_i y^i$ for some $a_i\in\C[x]$, $a_n\neq 0$, we have
\[\sum_{i=0}^n a_i(\alpha x)(g(x)+y)^i=\sum_{i=0}^n \alpha a_i(x) y^i.\] 
Hence $a_n(\alpha x)=\alpha a_n(x)$. By the assumption on $\alpha$ we get $a_n(x)=A_n x$ with $A_n\in\C^*$. 

If $n>0$, we also have $a_{n-1}(\alpha x)+n a_n(\alpha x)g(x)=\alpha a_{n-1}(x)$, that is,  we obtain
\begin{equation}\label{equality}
 a_{n-1}(\alpha x)+n \alpha A_n xg(x)=\alpha a_{n-1}(x).
\end{equation}  
Since $\aut(D)$ is a group, one may replace $\alpha$ with any of its powers, hence the degree of $a_{n-1}$ is necessarily 1. Since $a_{n-1}(0)=\alpha a_{n-1}(0)$ we have $a_{n-1}(x)=A_{n-1} x$, with $A_{n-1}\in\C^*$, but this is not compatible with (\ref{equality}) because $g\neq 0$. Then $n=0$ and $a(x,y)=A_0 x$ with $A_0\in\C^*$, which contradicts simplicity. We conclude that  $\alpha$  is a root of the unity.

By replacing $\rho$ with one of its powers, we may assume $\rho(x)=x$ and $\rho(y)=g(x)+y$ with $g(0)\neq 0$.  As before, if $n>0$ then (\ref{equality}) yields a contradiction, hence $n=0$, that is $a(x,y)=a(x)$. Since $D$ is simple, then  $a\in\C^*$; indeed, if $a=0$ the derivation $D$ stabilizes the ideal generated by $x$ and if $\deg a\geq 1$ it stabilizes the ideal generated by $a(x)$. Set $a=A\in\C^*$.

Now we use $\rho D(y)=D\rho(y)$ to obtain $b(x,y+g(x))=Ag'(x)+b(x,y)$. Write $b=\sum_{i=0}^m b_i y^i$, $b_i\in\C[x]$, $b_m\neq 0$; note that $b=0$ is not possible because $D$ is simple. If $m>1$, by arguing as in the case $n>0$ above we obtain a contradiction with $g\neq 0$. Then either $m=0$ and $b(x,y)=b(x)$ or $m=1$ and $b_1(x)g(x)=Ag'(x)$. Since $A\partial_x$ is a simple derivation on $\C[x]$, the former case contradicts 
 Shamsuddin's  criterion (see \cite{Sha} or \cite[Thm. 13.2.1]{Now}), whereas the latter is clearly not possible by degree reasons. 

Thus $\rho$ has finite order, which completes the proof.
\end{proof}

\noindent{\bf Proof of Theorem \ref{thm2}}.
Suppose $\aut(D)\neq \{id\}$ and take $\rho\in\aut(D)\setminus\{id\}$. By Corollary \ref{cor2} we know that $\rho$ has infinite order and none of its powers fix maximal ideals.

By a result of Lane (see \cite{Lan}) we know that $\rho$ stabilizes a nontrivial ideal $\ideala$. Then  a power $\rho^N$ of $\rho$ stabilizes the minimal associated primes of $\ideala$, which cannot be maximal ideals.  By replacing $\rho$ with $\rho^N$, we may suppose that $\rho$ stabilizes a prime ideal of height 1. Hence we assume there exists an irreducible polynomial $h\in\C[x,y]$ and an element $\mu\in\C^*$ such that $\rho(h)=\mu h$.

Note that a singular point of the curve $(h=0)$ corresponds to a maximal ideal which is fixed by a power of $\rho$, hence such a point cannot exist. 

From a result of Blanc and Stampfli (\cite[Theorem 2]{BS}) we deduce that there exists $\sigma\in\autp$ such that $\sigma(h)$ is either of the form $x$ or $x^ry^s-\lambda$, where $\lambda\in\C^*$, and $r,s$ are relatively prime positive integers. By replacing $\rho$ and $D$ with $\sigma\rho\sigma^{-1}$ and $\sigma D\sigma^{-1}$, respectively, we may assume $h=x$ or $h=x^ry^s-\lambda$. 

By Lemma \ref{lem7} the first possibility does not occur. Then assume $\rho(x^ry^s-\lambda)=\mu(x^ry^s-\lambda)$, with $\mu\in\C^*$; write $\rho(x)=\sum_{i=0}^n f_i y^i$ and $\rho(y)=\sum_{j=0}^m g_j y^j$, where $f_i,g_j\in\C[x]$, $i=0,\ldots,n$ and $j=0,\ldots,m$. Hence $f_n^rg_m^sy^{rn+sm}=\mu x^ry^s$, from which it follows $n=0, m=1$. Recalling that the Jacobian determinant associated to $\rho$ is a nonzero constant, we conclude $\rho(x)=\alpha x, \rho(y)=g_0+\beta y$ for suitable $\alpha,\beta\in\C^*$. Then we have
 \[(\alpha x)^r(g_0+\beta y)^s=\mu x^ry^s+\lambda(1-\mu),\]
which implies $g_0=0$ and $\rho$ fixes the maximal ideal $(x,y)$: contradiction. Hence the proof of Theorem \ref{thm2} is complete. 
\qed

\section{Automorphisms of Shamsuddin derivations}\label{sec2}

Let $D=\partial_x+(ay+b)\partial_y$ be a Shamsuddin derivation. Let $\rho\in\autp$. If $\rho(x)=f, \rho(y)=g$, then $\rho\in \aut(D)$ if and only if $f$ and $g$ verify
\begin{equation}\label{eq1}
\left\{\begin{array}{l}
\partial_x (f)+(ay+b)\partial_y( f)=1\\
\partial_x(g)+(ay+b)\partial_y(g)=\rho(a)g+\rho(b)
\end{array}\right.
\end{equation}
Writing  $f=f_0(x)+\ldots+f_n(x)y^n$ and  $g=g_0(x)+\ldots+g_m(x)y^m$ we obtain that (\ref{eq1}) is  equivalent to
\begin{equation}\label{eq2}
\left\{\begin{array}{l}
f_0'+bf_1+\sum_{i=1}^{n-1} (f_i'+iaf_i+(i+1)bf_{i+1})y^i+(f_n'+naf_n)y^n=1\\
g_0'+bg_1+\sum_{j=1}^{m-1} (g_j'+jag_j+(j+1)bg_{j+1})y^j+(g_m'+mag_m)y^m=\rho(a)g+\rho(b);
\end{array}\right.
\end{equation}
note that if we consider the polynomials $a,b, f$ and $g$  as polynomial functions, then we have $\rho(a)=a\circ f, \rho(b)=b\circ g$.
 
In the following example we treat the trivial case where $a=b=0$.

\begin{exa}
If $a=b=0$, then $\rho\in \aut(D)=\aut(\partial_x)$ if and only if there exist $P\in\C[y]$ and $d\in\C, \beta\in\C^*$ such that
\[\rho(x)=x+P(y), \rho(y)=d+\beta y,\]
i.e., $\aut(D)$ is a semidirect product ${\rm J_o}\rtimes(\C\ltimes\C^*)$, where $(\C\ltimes\C^*)$ has structure given by
\[(d_1,\beta_1)\cdot (d_2,\beta_2)=(d_1+d_2\beta_1,\beta_1\beta_2),\]
and 
\[{\rm J_o}=\{\sigma: \sigma(x)=x+P(y),\sigma(y)=y, P\in\C[y]\}\]
is the so-called \emph{de Jonqui\`eres} group.
\end{exa}

From now on we assume that a Shamsuddin derivation verifies $ay+b\neq 0$.

\begin{pro}\label{pro1}
Let $D=\partial_x+(ay+b)\partial_y$ be a Shamsuddin derivation with $a\neq 0$. We have the following assertions:

i) If $a\in\C^*, b=0$, then $\aut(D)=\K\times\K^*$

ii) If $a,b\in\K^*$, then $\aut(D)=\C\times (\C\ltimes \C^*)$. 

iii) If $\deg a\geq 1$ or $\deg b\geq 1$, then 
\[A(D)=\{\rho; \rho(x)=x, \rho(y)=g_0+dy, g_0'=ag_0+b(1-d), d\in\C^*\};\]
in particular, if $b=0$ we have $\aut(D)=\K^*$.
\end{pro}
\begin{proof}

First we note that $f=f_0=x+c$ and $g=g_0+dy$ for suitable $c\in\K, d\in\K^*$. Indeed, $n>0$ contradicts $f_n'+naf_n=0$ in the top equality in (\ref{eq2}); hence  $f=x+c$ for some $c\in\K$, and then $m\geq 1$ because $\rho$ is an automorphism. Moreover, the Jacobian determinant of $\rho$ is $\sum_{i=1}^m ig_i y^{i-1}$ and belongs to $\C^*$, from which the assertion follows.

Furthermore, from the bottom equality in (\ref{eq2}) we obtain
\begin{equation}\label{eq3}
g_0'+bd=a(x+c)g_0+b(x+c);
\end{equation}
here $a(x+c)$ (analogously $b(x+c)$) denotes the polynomial $a$, thought as a polynomial function, composed with $f=x+c$. 

Now, if $b=0$ we get $g_0':=D(g_0)=\rho(a)g_0$, whence $g_0=0$, because otherwise $\deg g_0'\geq \deg g_0$, which is not possible. We obtain  
\[\aut(D)=\{\rho; \rho(x)=x+c, \rho(y)=dy, c\in\C, d\in\C^*\},\]
and the assertion $i)$ follows easily.

On the other hand, since  $\rho^\ell\in\aut(D)$ for all $\ell\geq 1$, we deduce that the right hand side in both equalities of (\ref{eq3}) does not depend on $c$. Then $\deg a\geq 1$ or $\deg b\geq 1$ implies $c=0$. Indeed, if $c\neq 0$ the second of these implies $a\in\C$, and then $b\in\C$.

If  $\deg a\geq 1$ or $\deg b\geq 1$ (and $c=0$), from $g_0'+bd=ag_0+b$ it follows  $g_0'=ag_0+b(1-d)$, which proves $iii)$. 

Finally, assume $a,b\in\K^*$ and take  $\rho_1,\rho_2\in\aut(D)$. We know that $\rho_i(x)=x+c_i, c_i\in\K$, and $\rho_i(y)=\alpha_i+d_i y$, with $(\alpha_i,d_i)\in\C\times\C^*$. Then 
\[\rho_1\rho_2(x)=c_1+c_2,\ \ \rho_1\rho_2(y)=\alpha_2+d_2\alpha_1+d_1 d_2y.\]
If we consider the semidirect product structure $\C\ltimes \C^*$ with operation given by $(\alpha_1,d_1)\ltimes (\alpha_2,d_2)\mapsto (\alpha_2+d_2\alpha_1, d_1d_2)$ we obtain statement $ii)$, which completes the proof. 
\end{proof}

Note that the statement $iii)$ of the Proposition above does not say anything about whether $\aut(D)$ is trivial. Indeed, let us see two examples.  

\begin{exa}\label{exa4}
Take $a=x^2, b=x^5+x^4+x^3+x^2-2x+\epsilon$, for $\epsilon\in\C$; denote by $D_\epsilon$ the associated Shamsuddin derivation. As it follows from Proposition \ref{pro1}, part iii), we have $\rho(y)=g_0+dy$ with $\deg g_0\leq \deg b-\deg a$. A direct 
computation shows that the relation  $g'_0=ag_0+b(1-d)$ admits solutions with $d\neq 1$ if and only if $\epsilon=-1$. Furthermore, in this case we obtain solutions given by $g_0=g_e=-ex^3-ex^2-ex-4e$ and $d=1-e$, for $e\in\C\backslash\{1\}$. If we denote by $\rho_e\in\aut(D_{-1})$ the automorphism defined by $\rho_e(y)=g_e+(1-e)y$, then we have $\rho_e^{-1}=\rho_k$ with $k=e/(e-1)$. We deduce $\aut(D_{-1})\neq \{id\}$ and $\aut(D_\epsilon)=\{id\}$ if $\epsilon\neq -1$.
\end{exa}

\begin{exa}\label{exa5}
Consider the Shamsuddin  derivation $D=\partial_x+(2xy+x^3)\partial_y$; note $D$ stabilizes the ideal generated by $2y+x^2+1$. A straightforward computation using part $iii)$ of Proposition \ref{pro1} shows that $\aut(D)$ is the 1-parameter family of non-linear automorphisms $\rho_t$, $t\in\C^*$, defined by 
\[\rho_t(x)=x,\quad \rho_t(y)=\frac{(t-1)}{2}(x^2+1)+ ty.\] 
Note that this family is in fact a 1-parameter subgroup of $\aut(D)$.  
\end{exa}

\begin{cor}\label{cor3}
Assume $a\neq 0$. If $\deg a\geq 1$ or $\deg b\geq 1$, then $\aut(D)\neq \{id\}$ if and only if there exists $h\in\C[x]$  such that $D(h)=ah+b$.  In particular, if $b\neq 0$ one has $\deg b\geq \deg a$.
\end{cor}

\begin{proof}
Take $\rho\in\aut(D)\setminus\{id\}$. Part $iii)$ of Proposition \ref{pro1} implies $\rho\in \aut(D)$ if and only if there exists $g_0\in\C[x]$ and $d\in \K^*$ with $g_0'=ag_0+b(1-d)$ such that $\rho(x)=x$ and $\rho(y)=g_0+dy$. Since $\rho=id$ corresponds to $d\neq 1$, the first assertion follows by taking $h$ to be $(1-d)^{-1}g_0$. The rest of the proof is clear. 
\end{proof}

\noindent{\bf Proof of Theorem \ref{thm1}}. By  (\cite{Sha}),  $D$ is simple if and only if there is no $h\in\C[x]$ such that $D(h)=ah+b$. Moreover, by Theorem \ref{thm2} it suffices to prove the ``if'' part of the theorem. Thus, by Proposition \ref{pro1} we have $\deg a\geq 1$ or $\deg b\geq 1$, and the assertion follows from Corollary \ref{cor3}.

\qed
 
We finish the paper by giving an example of isotropy elements for Shamsuddin derivations with $a=0$ and $b\in\C^*$; we do not know how to treat the general case.

\begin{exa}\label{exa6}
Suppose $a=0$ and $b\in\C^*$. We are looking for elements in $\aut(D)$ with $n=0$, where $n$ is as in (\ref{eq1}). In this case we obtain, as before, $f=f_0=x+c$ from which it follows $m=1, g_1\in\C^*$ and (\ref{eq3}) becomes 
\[g_0'+bg_1=b.\]
An automorphism $\rho\in\aut(D)$ with $n=0$ is then defined by $\rho(x)=x+c$ and $\rho(y)=d+b(1-\beta)x+\beta y$ for a $\beta\in\C^*$.  There is then
a bijection between  $\C\times (\C\times\C^*)$ and such elements in $\aut(D)$ given by
\begin{equation}
(c,(d,\beta))\mapsto (\rho(x),\rho(y)).\label{eq3bis}
\end{equation}
In fact the subset of elements in $\aut(D)$ with $n=0$ is a subgroup isomorphic to a semidirect structure $\C\rtimes(\C\ltimes \C^*)$. Indeed, note that  for elements $\rho_1,\rho_2\in\aut(D)$ with
\[\rho_i(x)=x+c_i, \rho_i(y)=d_i+b(1-\beta_i)x+\beta_iy, i=1,2\]
we have $\rho_1\rho_2(x)=x+c_1+c_2$ and $\rho_1\rho_2(y)=d_2+\beta_2d_1+b(1-\beta_1\beta_2)x+\beta_1\beta_2y$; under the bijection (\ref{eq3bis}) 
the product $\rho_1\rho_2$ corresponds to $(c_1+c_2,(d_2+\beta_2 d_1, \beta_1\beta_2))$. Hence we have an exact sequence of groups
\[\xymatrix{\C\ar@{->}[r]&\C\rtimes(\C\ltimes \C^*)\ar@{->}[r]& \C\ltimes \C}\]
where the homomorphisms are, respectively, $c\mapsto (c,(0,1)), (c,(d,\beta))\mapsto (d,\beta)$, the semidirect product structure on the right-hand side being given by
\[(d_1,\beta_1)\rtimes (d_1,\beta_1)=(d_2+\beta_2 d_1,\beta_1\beta_2).\]

\end{exa}

\begin{tiny}
\noindent\address{Lu\'is Gustavo Doninelli Mendes, Instituto de Matem\'atica, UFRGS, Porto Alegre, BRASIL}
\email{mendes@mat.ufrgs.br}\\
\address{Ivan Pan, Centro de Matem\'atica, Facultad de Ciencias, UdelaR, Montevideo - URUGUAY}
 \email{ivan@cmat.edu.uy}
\end{tiny}
\end{document}